\documentclass[sn-mathphys-num]{sn-jnl}

\usepackage{graphicx}%
\usepackage{multirow}%
\usepackage{amsmath,amssymb,amsfonts}%
\usepackage{amsthm}%
\usepackage{mathrsfs}%
\usepackage[title]{appendix}%
\usepackage{xcolor}%
\usepackage{textcomp}%
\usepackage{manyfoot}%
\usepackage{booktabs}%
\usepackage{algorithm}%
\usepackage{algorithmicx}%
\usepackage{algpseudocode}%
\usepackage{listings}%
\usepackage{mathtools}
\usepackage{tikz}
\usepackage{amscd} 
\usetikzlibrary{cd}
\mathtoolsset{showonlyrefs}
\usetikzlibrary{calc}
\numberwithin{equation}{section}

\theoremstyle{thmstyleone}%
\newtheorem{theorem}{Theorem}[section]
\newtheorem{proposition}[theorem]{Proposition}%
\newtheorem{lemma}[theorem]{Lemma}%
\newtheorem{cor}[theorem]{Corollary}

\theoremstyle{definition}%
\newtheorem{definition}[theorem]{Definition}%
\newtheorem{example}[theorem]{Example}%
\newtheorem{remark}[theorem]{Remark}%

\def\vmf{\mathcal{M}_{\mathrm{vMF}}^n}
\def\econ{\nabla^{(\mathrm{e})}}
\def\mcon{\nabla^{(\mathrm{m})}}

\begin{document}

\title[Dually flat structure of von Mises-Fisher distributions]{Dually flat structure of von Mises-Fisher distributions}


\author{\fnm{Behzad} \sur{Najafi}, \fnm{Morteza} \sur{Sharifi-Najjar}, \fnm{Ryu} \fnm{Ueno}}


\abstract{The von Mises--Fisher (vMF) distribution is the natural exponential family analogue of the normal distribution on the unit hypersphere.
    We bring together previous results of its dually flat structure and establish new results on the information geometry of the family of vMF distributions.
    In particular, we characterize the exponential and mixture autoparallel submanifolds, and give the classification of doubly autoparallel submanifolds.
    Furthermore, we derive explicit formulas for the unique minimizers of the Kullback--Leibler divergence on these submanifolds and establish the corresponding projection theorems.}



\maketitle

\section{Introduction}\label{sec1}

We present a unified treatment of the dually flat structure of the \textit{von Mises--Fisher} (vMF) distribution and establish new results.
In particular, we characterize the submanifolds that are autoparallel with respect to both the exponential and mixture connections and establish explicit projection theorems for the Kullback--Leibler (KL) divergence onto these submanifolds.\\

Information geometry equips the manifold of probability distributions with a Riemannian structure through the Fisher information matrix. 
This simple idea has proven remarkably powerful: it allows one to bring the tools of differential geometry to bear on statistical inference, information theory, and machine learning \cite{Amari2000,Rao1945,Cover1999}.
In directional statistics, the vMF distribution \cite{Fisher1953,Mardia1999} plays a role analogous to that of the normal distribution on the hypersphere.
The vMF distribution also arises naturally as the conditional distribution of an isotropic multivariate normal restricted to the unit sphere \cite{Mardia1999}, and it is widely used in bioinformatics, geology, text mining, and machine learning on hyperspherical embeddings \cite{Dhillon2003,Banerjee2005}.

The Fisher information geometry of the vMF distribution has been studied from several directions. 
Hornik and Gr\"{u}n \cite{Hornik2013} derived the Fisher information matrix in the context of conjugate families and Jeffreys priors. 
Said, Bombrun, and Berthoumieu \cite{SaidBombrunBerthoumieu2019} proved a general theorem showing that the Fisher metrics of location-scale models on symmetric spaces, including the vMF model, are warped Riemannian metrics. \\

In this paper, we begin by developing the dually flat structure of the \textit{vMF manifold} (family of vMF distributions) and computing the Fisher metric and the Amari--Chentsov tensor. 
From this we characterize the exponential and mixture autoparallel submanifolds, and we show that the doubly autoparallel submanifolds are exactly the restrictions of the vMF manifold to linear subspaces (Theorem~\ref{daautoparallel}). 
On these submanifolds, we obatain an explicit form of the projection theorems of the KL divergence.\\

The paper is organized as follows. 
Section~\ref{sec2} reviews the preliminaries on information geometry, including dualistic structures, submanifolds, and divergences. 
Section~\ref{sec:vMF} develops the dually flat structure of the vMF manifold, computes its Fisher metric and Amari--Chentsov tensor, and derives the canonical divergence, an the canonical radial coordinate.
Section~\ref{sec:submanifolds} studies submanifolds of the vMF manifold, characterizing the autoparallel and doubly autoparallel cases and establishing projection theorems for the KL divergence.

\section{Preliminaries}\label{sec2}

In this section, we review the general theory of information geometry, which is the study of dualistic structures on a family of probability distributions.
The exposition in this section follows \cite{Amari2016}.
All objects in this paper are assumed to be smooth.
For a smooth manifold $M$, we denote by $\Gamma(E)$ the set of sections of a vector bundle $E$ over $M$.
In particular, we denote by $T^{(p,q)}M$ the tensor bundle of type $(p,q)$.

\subsection{Dualistic structure}
Let $\mathbb{R}^{n}$ be an affine space with the coordinate system $\theta = (\theta^1,\ldots,\theta^n)$.
We denote by $D$ the standard flat connection on $\mathbb{R}^n$, that is, $Dd\theta^i=0$.\\

\begin{definition}
	Let $M$ be a manifold with the global chart $\Theta\subset\mathbb{R}^n$ which is a convex open set, and $g$ a Riemannian metric on $M$.
	The triplet $(M,D,g)$ is called a \textit{Hessian manifold} if there exists a smooth function $\psi\in C^{\infty}(\Theta)$ such that $g=Dd\Psi$ where $\Psi = \psi\circ\theta$.\\
\end{definition}

The function $\Psi$ is called a \textit{potential} of $g$ with respect to $D$.
Since we have $Dg=DDd\Psi$, the affine connection $D$ is not necessarily the Levi-Civita connection of $g$.
Consequently, we obtain another affine connection $D^*$ on $M$ defined as follows.\\
\begin{definition}
	The \textit{dual connection} $D^*$ on a manifold $M$ of an affine connection $D$ with respect to a Riemannian metric $g$ is defined by
	\begin{align*}
		Xg(Y,Z) = g(D_X Y,Z) + g(Y,D^*_X Z),\quad X,Y,Z\in\Gamma(TM).
	\end{align*}
	We often call $(g,D,D^*)$ a \textit{dualistic structure} and denote by $(M,g,D,D^*)$ the manifold endowed with the dualistic structure.\\
\end{definition}

For a Hessian manifold $(M,D,g=Dd\Psi)$, its dual connection $D^*$ is also a flat connection.
In this case, we often call the dualistic structure $(g,D,D^*)$ a \textit{dually flat structure}.
An affine coordinate system $(\eta_1,\ldots,\eta_n)$ of $D^*$ is obtained by 
\begin{align}
	\label{expcor}
	\eta_i(x) = \frac{\partial \psi}{\partial \theta^i}(\theta(x)),\quad i=1,\ldots,n,\quad x\in M.
\end{align}
The coordinate system $(\eta_1,\ldots,\eta_n)$ is defined globally by the convexity of $\Theta$.
The two coordinate systems satisfy
\begin{align}
	\label{dualness}
	g\left(\frac{\partial}{\partial\theta^i},\frac{\partial}{\partial\eta_j}\right) = \delta_{ij}.
\end{align}
A potential of $g$ with respect to $D^*$ is obtained by 
\begin{align}
	\label{exppot}
	\varPhi(x) = \sum_{i=1}^n\theta^i(x)\frac{\partial\psi}{\partial \theta^i}(\theta(x)) - \psi(\theta(x)),\quad x\in M.
\end{align}
Therefore, if the range of $(\eta_1,\ldots,\eta_n)$ is convex, $(M,D^*,g)$ is also a Hessian manifold.\\

\begin{example}
	\label{expo}
	Let $\Omega\subset\mathbb{R}^k$ and $M=\{p_{\theta}\mid \theta\in\Theta\}$ be a family of probability distributions over $\Omega$ parametrized by a convex open set $\Theta\subset\mathbb{R}^n$.
	The family $M$ is called an exponential family if there exist functions $C,F_1,\ldots,F_n$ on $\Omega$ and $\psi\in C^{\infty}(\Theta)$ such that each $p_{\theta},\,\theta\in\Theta$ is expressed by
	\begin{align*}
		p_{\theta}(\omega) = \exp\left(C(\omega) - \sum_{i=1}^nF_i(\omega)\theta^i - \psi(\theta)\right),\quad \omega\in\Omega.
	\end{align*}
	The function $\psi$ is called the \textit{cumulant generating function}, which adjusts $p_{\theta}$ to be a probability distribution.
	We call $\theta=(\theta^1,\ldots,\theta^n)$ the \textit{natural coordinate} of $M$, and the standard affine connection $\nabla^{\mathrm{(e)}}$ on $\theta-$coordinate is called the \textit{exponential connection} of $M$.
	The \textit{Fisher metric} is defined on $M$ using the expected value by 
	\begin{align}
		g^F_{p_{\theta}}(X,Y) = \mathbb{E}_{p_{\theta}}[(X\log p_{\theta})(Y\log p_{\theta})],\quad X,Y\in\Gamma(TM),
	\end{align}
	which is equal to $g^F=\nabla^{\mathrm{(e)}}d\Psi$, where $\Psi = \psi\circ\theta$.
	We only consider the case when the Hessian of $\psi$ is positive definite so that $g^F$ is a Riemannian metric.
	The \textit{Amari--Chentsov} tensor is defined on $M$ by
	\begin{align}
		C^{AC}_{p_{\theta}}(X,Y,Z) = \mathbb{E}_{p_{\theta}}[(X\log p_{\theta})(Y\log p_{\theta})(Z\log p_{\theta})],\quad X,Y,Z\in\Gamma(TM),
	\end{align}
	which is just equal to $C^{AC} = \nabla^{\mathrm{(e)}}g^F$.
	The dual connection $\nabla^{(\mathrm{m})}$ of $(M,\nabla^{\mathrm{(e)}},g^F)$ is called the \textit{mixture connection} of $M$.
	Lastly, the coordinate system $\eta$ obtained by \eqref{expcor} is called the \textit{expectation coordinates}.
	Examples of exponential families that appear in information geometry are the probability simplex, normal distribution families, etc.\\
\end{example}

\begin{definition}
	Let $(g,D,D^*)$ be a dualistic structure on a manifold $M$.
	The \textit{difference tensor} $K$ is defined by 
	\begin{align}
		K_X Y = D_X Y - \nabla^g_X Y,\quad X,Y\in\Gamma(TM),
	\end{align}
	where $\nabla^g$ is the Levi-Civita connection of $g$.\\
\end{definition}

We recall the following well-known relations  from \cite{Amari2000}.\\
\begin{proposition}
    \label{difften}
    Let $(M, g, D, D^*)$ be a manifold equipped with a dually flat structure and difference tensor $K$. 
    For any vector fields $X, Y, Z \in \Gamma(TM)$, the following identities hold$\mathrm{:}$
    \renewcommand{\labelenumi}{$\operatorname{(\theenumi)}$}
    \begin{enumerate}
        \item $2\nabla^g_X Y = D_X Y + D^*_X Y$
        \item $2K_X Y = D_X Y - D^*_X Y$
        \item $-2g(K_X Y, Z) = (D_X g)(Y, Z)$\\
    \end{enumerate}
\end{proposition}

\begin{definition}
    Let $(g,D,D^*)$ be a dualistic structure on a manifold $M$ with $K$ the difference tensor.
    The \textit{$\alpha$-connections} $\nabla^{(\alpha)}$ of $(g,D,D^*)$ are defined by
    \begin{align}
        \label{alphacon}
        \nabla^{(\alpha)}_X Y = \nabla^g_X Y + \alpha K_X Y,\quad X,Y\in\Gamma(TM)
    \end{align}
    for each $\alpha\in\mathbb{R}$.\\
\end{definition}

\begin{proposition}\cite{Amari2000}
    The $\alpha$-connections on a dualistic structure $(g,D,D^*)$ have the following properties$\mathrm{:}$
    \renewcommand{\labelenumi}{$\operatorname{(\theenumi)}$}
    \begin{enumerate}
        \item $2\nabla^{(\alpha)}_X Y = (1+\alpha)D_X Y + (1-\alpha)D^*_X Y$,
        \item $(\nabla^{(\alpha)})^*=\nabla^{(-\alpha)}$,
        \item $(\nabla^{(\alpha)}_X g)(Y,Z) = \alpha(D_X g)(Y,Z)$,
    \end{enumerate}
    where $X,Y,Z\in\Gamma(TM)$.
    In particular, $\nabla^{(1)}=D$, $\nabla^{(-1)} = D^*$, and $\nabla^{(0)}= \nabla^g$.\\
\end{proposition}

\subsection{Submanifolds and the divergence}
Divergences are functions that indicate how close two probability distributions are.
Submanifolds of exponential families are important in information geometry due to their relation with the KL divergence.\\

Let $(M,D,g)$ be a Hessian manifold, $S\subset M$ a submanifold, and $\iota:S\to M$ the inclusion map.
We denote by $h=\iota^*g$ the first fundamental form.
The affine connection $D$ is projected on $M$ by
\begin{align}
	\label{submanifold}
	h(\nabla_X Y,Z) = g(D_X \iota_*Y, \iota_*Z),\quad X,Y,Z\in\Gamma(TS).
\end{align}

\begin{definition}
	The submanifold $S\subset M$ of a Hessian manifold $(M,D,g)$ is called a $D$-\textit{autoparallel} if the following holds:
	\begin{align*}
		\iota_*\nabla_X Y = D_{X}\iota_* Y,\quad X,Y\in\Gamma(TS).
	\end{align*}
\end{definition}

\begin{remark}
	Let $\theta = (\theta^1,\ldots,\theta^n)$ be the global convex $D$-affine coordinate system over $M$.
	Then all the $m$-dimensional $D$-autoparallel submanifolds of $M$ are open sets of some $m$-dimensional planes with respect to the $\theta-$coordinate.\\
\end{remark}

The following properties are known for submanifolds of Hessian manifolds \cite{Amari2000}.\\
\begin{proposition}
	Let $S\subset M$ be a submanifold of a Hessian manifold $(M,D,g)$, $h=\iota^*g$, and $\nabla$ the affine connection $D$ projected onto $S$ via $g$. 
	\renewcommand{\labelenumi}{$\operatorname{(\theenumi)}$}
	\begin{enumerate}
		\item If we denote by $\nabla^*$ the dual connection of $\nabla$ with respect to $h$, then we have
		\begin{align*}
			h(\nabla^*_X Y,Z) = g(D^*_X \iota_*Y,\iota_*Z),\quad X,Y,Z\in\Gamma(TS).
		\end{align*}
		\item If $S$ is a $D$-autoparallel submanifold of $(M,D,g)$, then $\nabla$ is also a flat connection on $S$.
		Consequently, $\nabla^*$ is also a flat connection.\\
		\item The submanifold $S$ is $D$-autoparallel if and only if for any $(p,v)\in TS$, there exists a $D$-geodesic $\gamma:(-\epsilon,\epsilon)\to M$ such that $\gamma(0)=p$, $\dot{\gamma}(0)=v$, and $\gamma(t)\in S$ for all $t\in(-\epsilon,\epsilon)$.\\
	\end{enumerate}
\end{proposition}

\begin{definition}
	Let $(g,D,D^*)$ be a dually flat structure on $M$ and $S\subset M$ be a submanifold.
	If $S$ is autoparallel with respect to both $D$ and $D^*$, then $S$ is called a \textit{doubly autoparallel submanifold} of $(M,g,D,D^*)$.\\
\end{definition}

\newpage
We define the Bregman divergence of a dually flat structure.\\

\begin{definition}
	Let $(g=Dd\Psi,D,D^*)$ be a dually flat structure on $M$ with the global convex $D$-affine coordinate system $(\theta^1,\ldots,\theta^n)$ on $M$.
Let us denote by \linebreak $\eta=(\eta_1,\ldots,\eta_n)$ the $D^*$-affine coordinate system and $\varPhi$ the potential function of $g$ with respect to $D^*$, each obtained by \eqref{expcor} and \eqref{exppot},  respectively.
	The \textit{Bregman divergence} $\rho(\cdot,\cdot)$ of $(g,D,D^*)$ is defined by
	\begin{align}
		\label{candi}
		\rho(x,y) = \Psi(y) + \varPhi(x) - \sum_{i=1}^{n} \theta^i(y)\eta_i(x),\quad x,y\in M.
	\end{align}
\end{definition}
\begin{example}
	For the dually flat structure $(g^F,\nabla^{(\mathrm{e})},\nabla^{(\mathrm{m})})$ obtained by an exponential family in Example~\ref{expo}, the Bregman divergence is equal to the KL divergence:
	\begin{align*}
		\mathrm{KL}(p\|q) = \int_{\Omega} p(\omega)\log\left(\frac{p(\omega)}{q(\omega)}\right)\,d\omega, \quad p,q\in M.
	\end{align*}
\end{example}

In general, the Bregman divergence is not necessarily a distance function since the symmetry and the triangle inequality do not have to hold.
The following property is known for the Bregman divergence \cite{Amari2000}.\\

\begin{proposition}
	Let $\rho(\cdot,\cdot)$ be the Bregman divergence of a dually flat structure $(g,D,D^*)$ on $M$.
	Then, for any $x,y\in M$, $\rho(x,y)\geq0$.
	In particular, we have $\rho(x,y)=0$ if and only if $x=y$.
\end{proposition}
The Bregman divergence of a dually flat structure $(g,D,D^*)$ satisfies 
\begin{align}
    \label{divdecom}
	\rho(x,z) = \rho(x,y) + \rho(y,z) + \sum_{i=1}^n(\theta^i(z)-\theta^{i}(y))(\eta_i(y) - \eta_i(x)),
\end{align}
for $x,y,z\in M$.
In particular, by \eqref{dualness} the generalized Pythagorean theorem holds for the Bregman divergence \cite{Amari2000}.\\

\begin{theorem}	\label{genpy}
	Let $\rho(\cdot,\cdot)$ be the Bregman divergence of a dually flat structure $(g,D,D^*)$ on $M$.
	For $x,y,z\in M$, suppose that there exists a $D^*$-geodesic $\gamma^*:[a,b]\to M$ from $\gamma^*(a)=x$ to $\gamma^*(b)=y$, and a $D$-geodesic $\gamma:[c,d]\to M$ from $\gamma(c)=y$ to $\gamma^*(d)=z$ such that $g_y(\dot{\gamma}(c),\dot{\gamma}^*(b))=0$.
	Then, the following equation holds$\mathrm{:}$
	\begin{align*}
		\rho(x,z) = \rho(x,y) + \rho(y,z).
	\end{align*}
\end{theorem}

Let $(g^F,\nabla^{(\mathrm{e})},\nabla^{(\mathrm{m})})$ be the dually flat structure obtained from the exponential family $M$.
Fix a point $x\in M$ and a submanifold $S\subset M$.
In information geometry, it is a common problem to find the point $y\in S$ that minimizes $\rho(x,y)$ or $\rho(y,x)$. \\

\begin{cor}\cite{Amari2000}
	\label{klproj}
	Let $\rho(\cdot,\cdot)$ be the Bregman divergence of a dually flat structure $(g,D,D^*)$ on $M$, $x\in M$, and $S\subset M$ a submanifold.
	\renewcommand{\labelenumi}{$\operatorname{(\theenumi)}$}
	\begin{enumerate}
		\item If $S$ is a $D$-autoparallel submanifold, then $\rho(x,y)=\mathrm{min}\{\rho(x,z)\mid z\in S\}$ if and only if there exists a $D^*$-geodesic from $x$ to $y\in S$ that is perpendicular to $S$ at $y$ with respect to $g$.\\
		\item If $S$ is a $D^*$-autoparallel submanifold, then $\rho(y,x)=\mathrm{min}\{\rho(z,x)\mid z\in S\}$ if and only if there exists a $D$-geodesic from $x$ to $y\in S$ that is perpendicular to $S$ at $y$ with respect to $g$.
	\end{enumerate}
\end{cor}

\section{The von Mises--Fisher distribution}
\label{sec:vMF}

Von Mises--Fisher (vMF) distributions are exponential distributions whose sample space is the unit hypersphere.
Let $\theta = (\theta^1,\ldots,\theta^n)$ be an affine coordinate system of $\mathbb{R}^n$.
Denote by $\langle\cdot,\cdot\rangle$ the Euclidean metric in the $\theta$-coordinate, and $\|\cdot\|$ the norm defined by it.\\

\begin{definition}
    Let $\mathbb{S}^{n-1} = \{x \in \mathbb{R}^n \mid \|x\| = 1\}$ be the unit hypersphere in the Euclidean space $\mathbb{R}^n$.
		The vMF distribution with mean direction $\mu \in \mathbb{S}^{n-1}$ and concentration parameter $\kappa \geq 0$ is 
		\begin{equation}\label{eq:vmf_density}
			p_{\mu,\kappa}(x) = C_n(\kappa)\exp(\kappa\langle x,\mu \rangle),\quad x\in \mathbb{S}^{n-1}.
		\end{equation}
		Here, the normalization constant $C_n(\kappa)$ is determined by the condition $\int_{\mathbb{S}^{n-1}} p_{\mu,\kappa}(x) \, dx = 1$, where $dx$ denotes the natural volume form of $\mathbb{S}^{n-1}$.\\
\end{definition}

Using the integral representation of the modified Bessel function, one obtains
	\begin{equation}\label{eq:vmf_normalization}
		C_n(\kappa) = \frac{\kappa^{n/2-1}}{(2\pi)^{n/2} I_{n/2-1}(\kappa)}, \qquad C_n(0)=\frac{1}{\operatorname{Vol}(\mathbb{S}^{n-1})},
	\end{equation}
where $I_\nu(z)$ is the modified Bessel function of the first kind, defined for $\nu > -1$ and $z \in [0, \infty)$:
	\begin{align*}
		I_\nu(z) = \sum_{k=0}^{\infty} \frac{(z/2)^{2k+\nu}}{k! \, \Gamma(\nu + k + 1)}.
	\end{align*}

Let  $p_{\mu,\kappa}$ be a vMF distribution.
Since $p_{\mu,\kappa}$ is independent of $\mu\in \mathbb{S}^{n-1}$ when $\kappa=0$, we can parametrize the vMF distribution $p_{\mu,\kappa}$ by $\theta = \kappa\cdot\mu\in\mathbb{R}^n$.
With this parameterization, we have 
\begin{align*}
    p_{\theta}(x) = C_n(\|\theta\|)\exp(\langle x,\theta\rangle),\quad x\in \mathbb{S}^{n-1}.
\end{align*}
If we let 
\begin{align*}
    \psi(\theta) = \log \int_{\mathbb{S}^{n-1}} \exp(\langle x,\theta\rangle)\,dx,\quad \theta\in\mathbb{R}^{n},
\end{align*}
then we have $C_n(\|\theta\|) = \exp(-\psi(\theta))$ and 
\begin{align*}
    p_{\theta}(x) = \exp(\langle x,\theta\rangle - \psi(\theta)),\quad x\in \mathbb{S}^{n-1}
\end{align*}
for each $\theta\in\mathbb{R}^n$.
Thus, we see that vMF distributions are exponential distributions, and $\theta$ is the natural coordinate.
We define $\Psi = \psi\circ\theta$ for future use.
We denote by $\mathcal{M}_{\mathrm{vMF}}^n = \{p_{\theta}\in C^{\infty}(\mathbb{S}^{n-1}) \mid \theta\in\mathbb{R}^n\}$ the manifold of all the vMF distributions on $\mathbb{S}^{n-1}$, and call it the \textit{vMF manifold}.\\

\begin{remark}
    For convenience in studying the geometry of submanifolds, we follow \cite{SaidBombrunBerthoumieu2019,Takano2010} and parametrize the vMF distribution by the natural parameter $\theta\in\mathbb{R}^n$ throughout this paper.
    However, it is also common to exclude the origin and use the parameter space $\mathbb{S}^{n-1}\times(0,\infty)$, which is diffeomorphic to $\mathbb{R}^n\setminus\{0\}$ via $(\mu,\kappa)\mapsto\kappa\mu$.
    The parameter $(\mu,\kappa)$ is called the \textit{polar coordinate}.
\end{remark}

\subsection{Dually flat structure of the vMF manifold}

We prepare some notations and equalities before computing the Fisher metric and the Amari-Chenstov tensor.
We use the polar coordinates
\begin{align*}
    \kappa(\theta) = \|\theta\|\in(0,\infty),\quad \mu(\theta) = \frac{\theta}{\|\theta\|}\in \mathbb{S}^{n-1}
\end{align*}
for $\theta\in\mathbb{R}^n\setminus\{0\}$ to give certain expressions.
For $\kappa\geq0$, we define 
\begin{align*}
    A_n(\kappa) = 
    \begin{cases} 
  \frac{I_{n/2}(\kappa)}{I_{n/2-1}(\kappa)}, & \text{if } \kappa > 0 \\ 
  0, & \text{if } \kappa = 0.
\end{cases}
\end{align*}
The function $A_n$ is called the \textit{mean resultant length function} or the \textit{ratio of modified Bessel functions}, which is a real analytic and strictly increasing fucntion (see \cite{Mardia1999,Watson1983}).\\
\begin{proposition}
    For the von Mises-Fisher manifold $\mathcal{M}^n_{\mathrm{vMF}}$, the Fisher metric $g^F$ and the Amari-Chentsov tensor $C^{AC}$ in polar coordniate are given by the following$\mathrm{:}$
    \begin{align*}
        g^F &= A^{'}_n(\kappa)d\kappa^2 + \kappa A_n(\kappa)g_{\mathbb{S}^{n-1}},\\
        C^{AC} &= A^{''}_n(\kappa)d\kappa^3 + \left(\kappa A^{'}_n(\kappa)-A_n(\kappa)\right)\mathrm{sym}(d\kappa\otimes g_{\mathbb{S}^{n-1}}),
    \end{align*}
    where $g_{\mathbb{S}^{n-1}}$ is the standard metric on the unit sphere, and
    \newpage
    \begin{align*}
        \mathrm{sym}(d\kappa\otimes g_{\mathbb{S}^{n-1}})(X,Y,Z) = d\kappa(X)g_{\mathbb{S}^{n-1}}(Y,Z)
         + d\kappa(Y)g_{\mathbb{S}^{n-1}}(Z,X) + d\kappa(Z)g_{\mathbb{S}^{n-1}}(X,Y)
    \end{align*}
    for $X,Y,Z\in\Gamma(T\mathcal{M}^n_{\mathrm{vMF}})$.
\end{proposition}
\begin{proof}
    In polar coordinate $(\mu,\kappa)$, we have
    \begin{align*}
        \frac{\partial}{\partial \theta^i}_{(\mu,\kappa)} = \frac{\theta^i}{\kappa}\frac{\partial}{\partial\kappa} + \frac{1}{\kappa}P_{\mu}\frac{\partial}{\partial \theta^i},
    \end{align*}
where $P_\mu$ is the projection of a vector $X \in T_{(\mu,\kappa)}\bigl(\mathbb{R}^n \setminus \{0\}\bigr)$ onto $T_\mu \mathbb{S}^{n-1}$, that is,
\begin{align*}
    P_{\mu}X = X - \langle X,\mu \rangle\mu.
\end{align*}
In particular, we have 
\begin{align}
    \label{eq1}
    g_{\mathbb{S}^{n-1}}\left(\frac{\partial}{\partial\theta^i}_{(\mu,\kappa)},\frac{\partial}{\partial\theta^j}_{(\mu,\kappa)}\right) &= \frac{1}{\kappa^2}\left\langle\frac{\partial}{\partial \theta^i}^{\perp},\frac{\partial}{\partial \theta^j}^{\perp}\right\rangle_{p_{(\mu,\kappa)}}\\
    &= \frac{1}{\kappa^2}\left(\delta_{ij} - \frac{\theta^i}{\kappa}\frac{\theta^j}{\kappa}\right).
\end{align}
We also have
\begin{align}
    d\kappa_{p_{\mu,\kappa}}\left(\frac{\partial}{\partial\theta^i}\right) = \frac{\theta^i}{\kappa}. \label{eq2}
\end{align}
We use the equality $g^F = \nabla^{\mathrm{(e)}}d\Psi$ and $C^{AC} = \nabla^{\mathrm{(e)}}g^F$.
From the formula
\begin{align*}
        \frac{d}{d\kappa}I_v(\kappa) = I_{v+1}(\kappa) + \frac{v}{\kappa}I_v(\kappa),
\end{align*}
we obtain
\begin{align*}
    \frac{\partial\psi}{\partial \theta^i} &= -\frac{\partial}{\partial \theta^i}\log(C_n(\kappa)) = \frac{A_n(\kappa)}{\kappa}\theta^i.
\end{align*}
As for the second and thrid derivative of $\psi$, from \eqref{eq1} and \eqref{eq2} we have
\begin{align*}
    \frac{\partial^2\psi}{\partial\theta^i\partial\theta^j} &= A^{'}_n(\kappa)\frac{\theta^i\theta^j}{\kappa^2} + A_n(\kappa)\left(\frac{\delta_{ij}}{\kappa} - \frac{\theta^i\theta^j}{\kappa^3}\right)\\
    &= A^{'}_n(\kappa)d\kappa\left(\frac{\partial}{\partial\theta^i}\right)d\kappa\left(\frac{\partial}{\partial\theta^j}\right) + \kappa A_n(\kappa)g_{\mathbb{S}^{n-1}}\left(\frac{\partial}{\partial\theta^i},\frac{\partial}{\partial\theta^j}\right),
\end{align*}
\begin{align*}
    \frac{\partial^3\psi}{\partial\theta^i\partial\theta^j\partial\theta^k} =& A^{''}_n(\kappa)\frac{\theta^i}{\kappa}\frac{\theta^j}{\kappa}\frac{\theta^k}{\kappa} + \left(\frac{A^{'}_n(\kappa)}{\kappa}-\frac{A_n(\kappa)}{\kappa^2}\right)\frac{\theta^i}{\kappa}\left(\delta_{jk} - \frac{\theta^j}{\kappa}\frac{\theta^k}{\kappa}\right)\\
    &+ \left(\frac{A^{'}_n(\kappa)}{\kappa}-\frac{A_n(\kappa)}{\kappa^2}\right)\left(\frac{\theta^j}{\kappa}\left(\delta_{ki} - \frac{\theta^k}{\kappa}\frac{\theta^i}{\kappa}\right) - \frac{\theta^k}{\kappa}\left(\delta_{ij} - \frac{\theta^i}{\kappa}\frac{\theta^j}{\kappa}\right)\right)\\
    =&A^{''}_n(\kappa)d\kappa\left(\frac{\partial}{\partial\theta^i}\right)d\kappa\left(\frac{\partial}{\partial\theta^j}\right)d\kappa\left(\frac{\partial}{\partial\theta^k}\right) \\
    &+ \left(\kappa A^{'}_n(\kappa)-A_n(\kappa)\right)\mathrm{sym}\left(d\kappa\otimes g_{\mathbb{S}^{n-1}}\right)\left(\frac{\partial}{\partial\theta^i},\frac{\partial}{\partial\theta^j},\frac{\partial}{\partial\theta^k}\right).
\end{align*}
\end{proof}

\vspace{\baselineskip}

\begin{remark}
    The Fisher metric and the Amari-Chenstov tensor are also given by K. Takano \cite{Takano2010}, expressed by the $\theta$-coordinate system.\\
\end{remark}

\begin{remark}
    The Fisher metric was also given by S. Said, L. Bombrun, and Y. Berthoumieu \cite{SaidBombrunBerthoumieu2019}.
    They also gave the limit at $\kappa = 0$, which is 
    \begin{align*}
        g^F_{ij}(0)=\left.\frac{\partial^2\psi}{\partial\theta_i\partial\theta_j}\right|_{\theta=0}=\frac{1}{n}\delta_{ij}.
    \end{align*}
\end{remark}

\vspace{\baselineskip}

\begin{proposition}
    The Amari-Chenstov tensor converges to $0$ at the origin.
\end{proposition}
\begin{proof}
    The cumulant generating function $\psi$ is an even function, that is, $\psi(-\theta) = \psi(\theta),\,\theta\in\mathbb{R}^n$ holds.
    Thus, the odd-order derivative of $\psi$ is always $0$ at the origin.
\end{proof}

\vspace{\baselineskip}

It is also easy to see that the Amari-Chenstov tensor vanishes when the dualistic structure is restricted on a sphere.\\

\begin{cor}
    Let $m<n$, and let $\mathcal{M}_{\mathbb{S}^{m}(r)} = \{p_{\theta}\mid \theta\in \mathbb{S}^m(r)\}$.
    For the dualistic structure $(h,\nabla,\nabla^*)$ induced on $\mathcal{M}_{\mathbb{S}^{m}(r)}$ by \eqref{submanifold}, we have $\nabla h=0$, which is equivalent to $\nabla = \nabla^* = \nabla^h$.
\end{cor}

\begin{proof}
    Since $\kappa$ is constant on $\mathcal{M}_{\mathbb{S}^{m}(r)}$, we get $\mathrm{d}\kappa = 0$ on the subsphere, which causes all components of the difference tensor $K$ (and the cubic form $C^{\mathrm{AC}}$) to vanish on spherical tangent vectors.
\end{proof}

\newpage

\begin{proposition}
    The difference tensor in polar coordinate system $(\mu,\kappa)$ is
    \begin{align*}
        K_{(\mu,\kappa)} =& -\frac{1}{2}\frac{A_n^{''}(\kappa)}{A_n^{'}(\kappa)}\frac{\partial}{\partial \kappa}\otimes d\kappa^2 + \frac{1}{2}\left(\frac{A_n(\kappa)}{A^{'}_n(\kappa)}-\kappa\right)\frac{\partial}{\partial\kappa}\otimes g_{\mathbb{S}^{n-1}}\\
        &+ \frac{1}{2}\left(\frac{1}{\kappa}-\frac{A_n^{'}(\kappa)}{A_n(\kappa)}\right)\mathrm{sym}(P_{\mu}\otimes d\kappa).
    \end{align*}
\end{proposition}

\begin{proof}
    The proof is a simple computation using equation $(3)$ of Proposition~\ref{difften}.
\end{proof}

\vspace{\baselineskip}

We now express the affine connections on the vMF manifold by using the mean direction and concentration.\\

\begin{proposition}
    The $\alpha$-connections of the von Mises-Fisher manifold on the polar coordinate $(\mu,\kappa)$ are given by
    \begin{align*}
        \nabla^{(\alpha)}_{\frac{\partial}{\partial \kappa}} \frac{\partial}{\partial \kappa} &= (1-\alpha)\frac{1}{2}\frac{A^{''}_n(\kappa)}{A^{'}_n(\kappa)}\frac{\partial}{\partial \kappa},\\
        \nabla^{(\alpha)}_{\frac{\partial}{\partial \kappa}} X &= \frac{1}{2}\left((1 + \alpha)\frac{1}{\kappa}+(1 - \alpha)\frac{A^{'}_n(\kappa)}{A_n(\kappa)}\right)X,\\
        \nabla^{(\alpha)}_X Y &= \nabla^{g_{\mathbb{S}^{n-1}}}_X Y + \frac{1}{2}\left((\alpha-1)\frac{A_n(\kappa)}{A^{'}_n(\kappa)} - (1 + \alpha)\kappa\right)g_{\mathbb{S}^{n-1}}(X,Y)\frac{\partial}{\partial\kappa},
    \end{align*}
    where $X,Y\in \Gamma(T\mathbb{S}^{n-1})$.\\
\end{proposition}

\begin{proof}
    The Fisher metric of the vMF family admits a warped-product representation; see \cite{SaidBombrunBerthoumieu2019}.
    By applying the standard formulas for the Levi-Civita connection of a warped product \cite{ONeill1983}, we obtain
    \begin{align*}
        \nabla^{g^F}_{\frac{\partial}{\partial \kappa}} \frac{\partial}{\partial \kappa} &= \frac{1}{2}\frac{A^{''}_n(\kappa)}{A^{'}_n(\kappa)}\frac{\partial}{\partial \kappa},\\
        \nabla^{g^F}_{\frac{\partial}{\partial \kappa}} X &= \frac{1}{2}\left(\frac{1}{\kappa}+\frac{A^{'}_n(\kappa)}{A_n(\kappa)}\right)X,\\
        \nabla^{g^F}_X Y &= \nabla^{g_{\mathbb{S}^{n-1}}}_X Y - \frac{1}{2}\left(\kappa + \frac{A_n(\kappa)}{A^{'}_n(\kappa)}\right)g_{\mathbb{S}^{n-1}}(X,Y)\frac{\partial}{\partial\kappa},
    \end{align*}
    where $X,Y\in \Gamma(T\mathbb{S}^{n-1})$.
    The definition \eqref{alphacon} of $\alpha$-connections completes the proof.
\end{proof}

\newpage

\subsection{Canonical divergence of the vMF manifold}

Since the vMF manifold is an exponential family, the Bregman divergence is the Kullback-Leibler divergence.
We begin by computing the expectation coordinate.
By \eqref{expcor}, the expectation coordinate $\eta=(\eta_1,\ldots,\eta_n)$ of $\vmf$ without the origin is
\begin{align*}
    \eta_i(p_{\theta}) = \frac{A_n(\kappa)}{\kappa}\theta^i,\quad i=1,\ldots,n,\quad p_{\theta}\in\vmf\setminus\{p_0\}.
\end{align*}
As $\kappa\to0$, the function $A_n(\kappa)$ is order $\kappa$, that is,
\begin{align*}
    A_n(\kappa) = \frac{\kappa}{n} + O(\kappa^3),\quad \mathrm{as}\quad\kappa\to 0.
\end{align*}
Therefore, the limit $\lim_{\kappa\to 0} \frac{A_n(\kappa)}{\kappa} = \frac{1}{n}$ holds \cite{Watson1983}, and concequently we have $\eta_i(p_{0}) = 0$.
With the following lemma, the range of the expectation coordinate $\eta$ is $B_n^{\circ}(1)$, the open ball of range $1$ with center $0$.\\

\begin{lemma}\label{lem:An_strict_monotonicity}
For any dimension $n \ge 2$, the concentration ratio
function
\[
A_n(\kappa)
= \frac{I_{n/2}(\kappa)}{I_{n/2-1}(\kappa)}
\]
is strictly increasing on $[0,\infty)$ with
\[
\lim_{\kappa\to\infty} A_n(\kappa)=1.
\]
Consequently, $A_n:[0,\infty)\to[0,1)$ is a real-analytic diffeomorphism, and its inverse $A_n^{-1}:[0,1)\to[0,\infty)$ is strictly increasing and real-analytic.
\end{lemma}

\begin{proof}
The assertions follow from the known properties of ratios of modified Bessel functions (see \cite[Sections 1 and 2]{HornikGrun2014}), together with the real-analytic inverse function theorem.
\end{proof}

\vspace{\baselineskip}

\begin{proposition}
    For $\theta_1 = (\mu_1,\kappa_1),\theta_2 = (\mu_2,\kappa_2)\in\mathbb{R}^n$, the Kullback-Leibler divergence between $p_{\theta_1},p_{\theta_2}\in\mathcal{M}^n_{\mathrm{vMF}}$ is
    \begin{align}
        \label{vmfkl}
        \mathrm{KL}(p_{\theta_1}\|p_{\theta_2}) &= \psi(\theta_2) - \psi(\theta_1) + A_n(\kappa_1)(\kappa_1-\kappa_2\langle\mu_1,\mu_2\rangle)\\
        &= \log\frac{C_n(\kappa_1)}{C_n(\kappa_2)} + A_n(\kappa_1)(\kappa_1-\kappa_2\langle\mu_1,\mu_2\rangle).
    \end{align}
\end{proposition}

\begin{proof}
    We begin by computing the potential $\varPhi$ of $g^F$ with respect to $\mcon$.
    By \eqref{exppot}, we obtain
    \begin{align*}
        \varPhi(p_{\theta}) = \kappa A_n(\kappa) - \psi(\theta),\quad p_{\theta}\in\vmf.
    \end{align*}
    Since the Kullback Leibler divergence is the Bregman divergence~\eqref{candi} on exponential families, we obtain the desired equation.
\end{proof}

If the two vMF distributions $p_{\theta_1},p_{\theta_2}$ have the same concentration $\kappa$, then the Kullback-Leibler divergence is symmetric, and is equal to
\begin{align*}
    \mathrm{KL}(p_{\theta_1}\|p_{\theta_2}) = \kappa A_n(\kappa)(1-\langle\mu_1,\mu_2\rangle).
\end{align*}
In the last term, we have $\langle\mu_1,\mu_2\rangle = \cos(d_{\mathbb{S}^{n-1}}(\mu_1,\mu_2))$, where $d_{\mathbb{S}^{n-1}}$ is the distance function of $(\mathbb{S}^{n-1},g_{\mathbb{S}^{n-1}})$.\\

On the other hand, if $p_{\theta_1},p_{\theta_2}$ have the same mean direction $\mu$, then the Kullback-Leibler divergence is 
\begin{align*}
    \mathrm{KL}(p_{\theta_1}\|p_{\theta_2}) = \log\frac{C_n(\kappa_1)}{C_n(\kappa_2)} + A_n(\kappa_1)(\kappa_1-\kappa_2).
\end{align*}

\bigskip

\section{Submanifolds of the vMF manifold and relations with the KL divergence}
\label{sec:submanifolds}
Identifying $\econ$ and $\mcon$-autoparallel submanifolds of exponential families is important because the KL divergence minimization is obtained by projection and Pythagorean characterizations.

\subsection{Autoparallel submanifolds}

The $\nabla^{\mathrm{(e)}}$-autoparallel and $\nabla^{(\mathrm{m})}$-autoparallel submanifolds of $\mathcal{M}^n_{\mathrm{vMF}}$ can be determined using their global affine coordinates.\\
\begin{proposition}
    \label{eautoparallel}
    A submanifold $M\subset\vmf$ is a $\nabla^{\mathrm{(e)}}$-autoparallel submanifold if and only if there exists a plane $P\subset\mathbb{R}^n$ such that $M$ is an open set of $\mathcal{M}_{P}=\{p_{\theta}\in\mathcal{M}_{\mathrm{vMF}}^n\mid \theta\in P\}$.\\
\end{proposition}
\begin{proof}
    It follows from the fact that the natural coordinate $\theta$ is an affine coordinate system of $\nabla^{(\mathrm{e})}$.
\end{proof}

\vspace{\baselineskip}


For convenience, we define a function $f_n:[0,1)\to[n,\infty)$ by
\begin{align*}
        f_n(r) =
        \begin{cases}
        \frac{A^{-1}_n(r)}{r}, & \text{if } 0 < r < 1 \\ 
        n, & \text{if } r = 0.
        \end{cases}
\end{align*}

\begin{proposition}
    \label{mautoparallel}
    A submanifold $M\subset\vmf$ is a $\mcon$ autoparallel submanifold if and only if there exists $a\in B^{\circ}_n(1)$ and a linear subspace $V\in\mathbb{R}^n$ with $\langle a,v \rangle = 0$ for all $v\in V$, such that $M$ is an open set of $\mathcal{M}_{H}=\{p_{\theta}\in\vmf\mid \theta\in H\}$, where
    \begin{align}
        \label{mplane}
        H = \left\{F(v) \mid v\in V\cap B_n^{\circ}(\sqrt{1-\|a\|^2}) \right\},
    \end{align}
    and 
    \begin{align}
        \label{inversemap}
        F(v)=f_n(\sqrt{\|a\|^2+\|v\|^2})(a+v),\quad v\in V\cap B_n^{\circ}(\sqrt{1-\|a\|^2}).
    \end{align}
\end{proposition}
\begin{proof}
    On the $\eta$ coordinate, a $\mcon$-autoparallel manifold $M$ is an open subset of some planes $P$ in $\mathbb{R}^n$ restricted on $B^{\circ}_n(1)$. 
    Therefore, there exists $P=\{a+v\mid v\in V\}$ where $a\in B^{\circ}_n(1)$ and a linear subspace $V\subset\mathbb{R}^n$ with $\langle a,v \rangle=0$ for all $v\in V$, such that $M$ is an open set of $P\cap B^{\circ}_n(1) = \{a+v\mid  v\in B^{\circ}_n(\sqrt{1-\|a\|^2})\}$.
    The rest follows from the expression of the natural coordinate by
    \begin{align}
        \label{exptonat}
        \theta^i = f_n(\|\eta\|)\eta_i,\quad i=1,\ldots,n,
    \end{align}
    using the expectation coordinate.
\end{proof}

\vspace{\baselineskip}


With Propositions \ref{eautoparallel} ans \ref{mautoparallel}, we obtain the doubly autoparallel submanifolds of the vMF manifold.\\

\begin{theorem}
    \label{daautoparallel}
    Let $M^m\subset\vmf$ be a connected, posistive--dimensional submanifold.
    Then $M$ is doubly autoparallel if and only if there exists a linear subspace $V\subset\mathbb{R}^n$ such that $M$ is an open set of $\mathcal{M}_V = \{p_{\theta}\in\vmf\mid\theta\in V\}$.\\
\end{theorem}

\begin{proof}
Let $M^m$ be a connected, positive-dimensional doubly autoparallel submanifold of $\vmf$.
Since $M$ is a $\mcon$-autoparallel submanifold,
there exist $a\in B_n^\circ(1)$ and an $m$-dimensional
linear subspace $V\subset\mathbb{R}^n$ with
$\langle a,v\rangle=0$ for all $v\in V$
such that $M$ is an open submanifold of $\mathcal{M}_H$,
where $H$ is defined by \eqref{mplane}.
Since $M$ is also a $\econ$-autoparallel submanifold,
there exist $b\in\mathbb{R}^n$ and an $m$-dimensional
linear subspace $W\subset\mathbb{R}^n$ with
$\langle b,w\rangle=0$ for all $w\in W$
such that
\[
  M\subset\{p_{\theta}\in\vmf\mid \theta\in P^{'}\}
\]
where $P^{'}=\{b+w\mid w\in W\}$.
We prove that $V=W$ and $a=b=0$.
Put
\[
  D=V\cap B_n^\circ(\sqrt{1-\|a\|^2})
\]
and define $F:D\to\mathbb{R}^n$ by \eqref{inversemap}.
Since $F$ is the restriction of the inverse of the expectation-coordinate map, it is real analytic on $D$.
As $M$ is an open submanifold of $\mathcal{M}_H$, there exists a nonempty open subset $U\subset D$ such that $M=\{p_{F(v)}\mid v\in U\}$.
In particular, we have $F(v)\in P^{'}$ for every $v\in U$.
Denoting by $P_{W^\perp}$ the orthogonal projection
onto $W^\perp$, we have
\[
  P_{W^\perp}F(v)=b,\qquad v\in U.
\]
Since $D$ is connected and $F$ is real analytic, the identity theorem for real analytic functions implies that
\[
  P_{W^\perp}F(v)=b,\qquad v\in D,
\]
and consequently, $F(D)\subset P^{'}$.
For any $x,y\in P^{'}$, we have $x-y\in W$.
Therefore, for any $v\in D$ we have
\begin{align*}
  F(v)-F(-v)
  &=f_n(\sqrt{\|a\|^2+\|v\|^2})(a+v)-f_n(\sqrt{\|a\|^2+\|-v\|^2})(a-v)\\
  &=2f_n(\sqrt{\|a\|^2+\|v\|^2})v \in W.
\end{align*}
Since $f_n$ takes positive values, we obtain $V\subset W$. 
As $\dim V=\dim W=m$, it follows that $V=W$.
On the other hand, for any $x,y\in P^{'}$ it holds that $\frac{1}{2}(x+y)\in P^{'}$, thus for $v\in V\cap B^{\circ}_n(\sqrt{1-\|a\|^2})$, we have 
    \begin{align*}
        f_n(\sqrt{\|a\|^2+\|v\|^2})a&=\frac{1}{2}\left(f_n(\sqrt{\|a\|^2+\|v\|^2})(a+v) + f_n(\sqrt{\|a\|^2+\|-v\|^2})(a-v)\right)\in P^{'}.
    \end{align*}
which implies that $f_n(\sqrt{\|a\|^2+\|v\|^2})a = b$ for any $v\in V\cap B^{\circ}_n(\sqrt{1-\|a\|^2})$, from the projection on $V^{\perp}$.
Suppose that $a\neq0$.
Since $\dim V=m\ge1$, we can choose a sequence
$\{v_k\}\subset D$ such that
\[
  \|v_k\|\longrightarrow\sqrt{1-\|a\|^2},\quad \mathrm{as}\,k\longrightarrow\infty.
\]
Then
\[
  \sqrt{\|a\|^2+\|v_k\|^2}\longrightarrow1,\quad \mathrm{as}\,k\longrightarrow\infty.
\]
Since $f_n(r)\to\infty$ as $r\to1$,
we obtain
\[
  \|f_n(\sqrt{\|a\|^2+\|v_k\|^2})a\|,\quad \mathrm{as}\,k\longrightarrow\infty,
\]
which contradicts
\[
  f_n(\sqrt{\|a\|^2+\|v_k\|^2})a=b.
\]
Therefore, $a=0$, and hence $b=0$.
Thus $P~{'}=V$, and $M$ is an open submanifold of $\mathcal{M}_V$.\\

Conversely, for any linear subspace $V\subset\mathbb{R}^n$, 
The submanifold $\mathcal{M}_V$ is clearly a $\econ$-autoparallel submanifold.
    This submanifold in the expectation coordinate is 
    \begin{align*}
        \eta(\mathcal{M}_V) = V\cap B_n^{\circ}(1),
    \end{align*}
    which indicates that $\mathcal{M}_V$ is also a $\mcon$-autoparallel submanifold.
\end{proof}

\vspace{\baselineskip}

\begin{remark}
    From the definition of $\alpha$-connections \eqref{alphacon}, very doubly autoparallel submanifold of the vMF manifold is also autoparallel with respect to the $\alpha$-connection $\nabla^{(\alpha)}$ for all $\alpha\in\mathbb{R}$.
\end{remark}

\subsection{Projection theorem of the KL divergence on doubly autoparallel submanifolds}

Let $V\subset\mathbb{R}^n$ be a linear subspace.
The submanifold $\mathcal{M}_{V}\subset\vmf$ is doubly autoparallel by Theorem \ref{daautoparallel}.
For a vMF distribution $p_{\theta}\in\vmf$, we can consider two complementary projection problems$:$
identifying the point $p_{\xi}\in\mathcal{M}_V$ that minimizes $\mathrm{KL}(\cdot\|p_{\theta})$ (the $e$-projection) or $\mathrm{KL}(p_{\theta}\|\cdot)$ (the $m$-projection), whenever the minimum is attained.\\

\begin{proposition}
    \label{eproj}
    Let $V\in\mathbb{R}^n$ be a linear subspace.
    For an arbitrary $p_{\theta}\in\vmf$, the vMF distribution $p_{\xi}\in \mathcal{M}_{V}$ where $\xi=P_{V}\theta$ is the unique point where $\mathrm{KL}(\cdot\|p_{\theta})$ attains the minimum.
\end{proposition}

\begin{proof}
    For any $v\in V$ and $w\in T_{p_{v}}\mathcal{M}_{V} \simeq V$, we choose a curve $\gamma:(-\epsilon,\epsilon)\to \mathcal{M}_{V}$ such that $\gamma(0)=v$, and $\frac{d}{dt}(\eta\circ\gamma)(0) = w$.
    Then, from the definition of the canonical divergence \eqref{candi}, we have
    \begin{align*}
        \left.\frac{d}{dt}\mathrm{KL}(p_{\gamma(t)}\|p_{\theta})\right|_{t=0} &= \left.\frac{d}{dt}\left(\varphi(\eta\circ\gamma(t)) - \langle\theta,\eta\circ\gamma(t)\rangle\right)\right|_{t=0}\\
        &=\sum_{i=1}^n\frac{\partial \varphi}{\partial \eta_i}(\eta(v))\frac{\eta_i\circ\gamma}{dt}(0) - \langle\theta,w\rangle\\
        &=\langle v,w\rangle - \langle\theta,w\rangle\\
        &=\langle v-\theta,w\rangle.
    \end{align*}
    Therefore, $p_{\xi}\in\mathcal{M}_{V}$ is a critical point of the function $\mathrm{KL}(\cdot\|p_{\theta})$ if and only if the vector $\xi-\theta$ is perpendicular to $V$, which gives us $\xi=P_{V}\theta$.
    It remains to prove that $p_\xi$ is the unique
global minimizer.
By \eqref{divdecom}, for any $v\in V$, we have
\begin{align*}
    \mathrm{KL}(p_v\|p_\theta)
    &=
    \mathrm{KL}(p_v\|p_\xi)
    +\mathrm{KL}(p_\xi\|p_\theta)
    +\langle\xi-\theta,
    \eta(p_v)-\eta(p_\xi)\rangle\\
    &=
    \mathrm{KL}(p_v\|p_\xi)
    +\mathrm{KL}(p_\xi\|p_\theta),
\end{align*}
where the last equality follows from
$\xi-\theta\in V^\perp$ and
$\eta(p_v)-\eta(p_\xi)\in V$.
Since $\mathrm{KL}(p_v\|p_\xi)\geq 0$,
with equality if and only if $v=\xi$,
the point $p_\xi$ is the unique global minimizer
of $\mathrm{KL}(\cdot\|p_\theta)$ on $\mathcal{M}_V$.
    
\end{proof}

\newpage

Since doubly autoparallel submanifolds of the vMF manifolds are so simple, the next Corollary follows.\\

\begin{cor}
    \label{klcor}
    Let $V\subset\mathbb{R}^n$ be a linear subspace.
    For any perpendicular line $\{tv + (1-t)\theta\mid t\in\mathbb{R}\}$ of $V$ in $\mathbb{R}^n$, where $\theta\in\mathbb{R}^n\setminus V$ and $v = P_V\theta$, we have
    \begin{align*}
        g^F_{p_{v}}\left(v-\theta,w\right) = 0,\quad w\in T_{p_{v}}\mathcal{M}_{V}\simeq V
    \end{align*}
    for $v-\theta\in T_{p_{v}}\vmf\simeq\mathbb{R}^n$.
\end{cor}

\begin{proof}
    From Proposition~\ref{eproj}, the function $\mathrm{KL}(\cdot\|p_{\theta})$ attains its minimum at $p_{v}$ on $V$.
    The curve $p_{\gamma(t)}$ where $\gamma(t) = tv + (1-t)\theta,\,t\in\mathbb{R}$ is the $\nabla^{\mathrm{(e)}}$-geodesic from $p_{\theta}$ to $p_{v}$, unique up to affine parameterization, so by Corollary~\ref{klproj} the vector $\left.\frac{d}{dt}p_{\gamma(t)}\right|_{t=1} = v-\theta$ is perdicular to $T_{p_{v}}\mathcal{M}_V$ with respect to $g^F$.
\end{proof}

\vspace{\baselineskip}

\begin{proposition}
    \label{mproj}
    Let $V\in\mathbb{R}^n$ be a linear subspace.
    For an arbitrary $p_{\theta}\in\vmf\setminus\mathcal{M}_V$, the vMF distribution $p_{\xi}\in\mathcal{M}_{V}$ is the unique point on $\mathcal{M}_{V}$ such that $\displaystyle \mathrm{KL}(p_{\theta}\|p_{\xi}) = \mathrm{min}\{\mathrm{KL}(p_{\theta}\|p_{v})\mid v\in V\}$, where
    \begin{align*}
        \xi = A_p^{-1}\left(\frac{A_p(\|\theta\|)}{\|\theta\|}\|P_V\theta\|\right)\frac{P_V\theta}{\|P_V\theta\|},
    \end{align*}
    if $P_{V}\theta \neq 0$.
    Otherwise, $\xi=0$.
\end{proposition}

\begin{proof}
    Again, for any $v\in V$ and $w\in T_{p_{v}}\mathcal{M}_V \simeq V$, we define a curve $\gamma(t) = v + tw,\,t\in\mathbb{R}$.
    From the definition of the canonical divergence \eqref{candi}, we have
    \begin{align*}
        \left.\frac{d}{dt}\mathrm{KL}(p_{\theta}\|p_{v + tw})\right|_{t=0} &= \left.\frac{d}{dt}\left(\psi(v + tw) - \langle v + tw,\eta(p_{\theta})\rangle\right)\right|_{t=0}\\
        &= \sum_{i=1}^n\frac{\partial\psi}{\partial\theta^i}(v)\frac{d\theta^i\circ\gamma}{dt}(0) - \langle w,\eta(p_{\theta})\rangle\\
        &= \langle\eta(p_v)-\eta(p_{\theta}),w\rangle.
    \end{align*}
    Here, since $\eta(p_v)\in V\cap B_n^{\circ}(1)$, $p_{\xi}\in\vmf$ gives the extremum of $\mathrm{KL}(p_{\theta}\|\cdot)$ if and only if $\eta(p_{\xi}) = P_{V}\eta(p_{\theta})$.
    The minimality of $\mathrm{KL}(p_{\theta}\|p_{\xi})$ is obtained by a similar proof of Proposition \ref{eproj}.
    Lastly, we have
    \begin{align*}
        \eta(p_{\xi}) &= P_{V}\eta(p_{\theta})\\
        &= \frac{A_n(\|\theta\|)}{\|\theta\|}P_{V}\theta.
    \end{align*}
    The expression \eqref{exptonat} gives us the desired formula of $\xi$.
\end{proof}

\vspace{\baselineskip}

\subsection{Application}
The projection by $\econ$-geodesic on doubly autoparallel manifolds in the vMF manifold has a simple form.
Thus, we are able to decompose the KL divergence $\mathrm{KL}(p_{\theta_1}\|p_{\theta_2})$ into parallel and perpendicular components.\\
\begin{proposition}
    \label{edecom}
    Let $p_{\kappa_1\mu_1},p_{\kappa_2\mu_2}\in\vmf$ with $\kappa_1\neq 0$.
    For $p_{l}\in\vmf$ where $l = \kappa_2\langle\mu_1,\mu_2\rangle\mu_1$, we have
    \begin{align*}
        \mathrm{KL}(p_{\kappa_1\mu_1}\|p_{\kappa_2\mu_2}) = \mathrm{KL}(p_{\kappa_1\mu_1}\|p_{l}) + \mathrm{KL}(p_{l}\|p_{\kappa_2\mu_2}).
    \end{align*}
    Particularly, if $l = \kappa_1\mu_1$, we have
    \begin{align*}
        \mathrm{KL}(p_{\kappa_1\mu_1}\|p_{\kappa_2\mu_2}) = \log\frac{C_n(\kappa_1)}{C_n(\kappa_2)}.
    \end{align*}
\begin{figure}[htbp]
    \centering
    \includegraphics[width=0.85\textwidth]{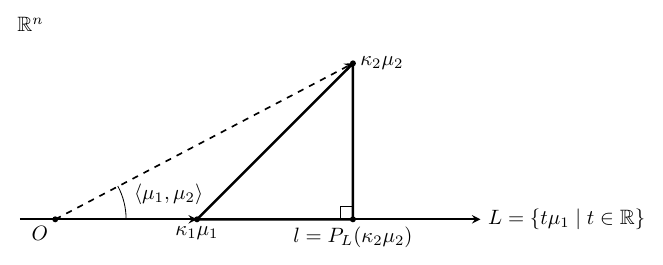}
    \caption{Geometric picture of the decomposition of the KL divergence in Proposition~\ref{edecom}:
the projection $l=P_L(\kappa_2\mu_2)$ of $\kappa_2\mu_2$ onto the line $L=\{t\mu_1\mid t\in\mathbb{R}\}$, and the right angle at $l$ with respect to $g^F$.}
    \label{fig:kl_decomposition}
\end{figure}
\end{proposition}

\begin{proof}
    Define a line $L\subset\mathbb{R}^n$ by $L=\{t\mu_1\mid t\in\mathbb{R}\}$.
    The submanifold $\mathcal{M}_{L} = \{p_v\mid v\in L\}\subset\vmf$ is doubly autoparallel, so the function $\mathrm{KL}(\cdot\|p_{\kappa_2\mu_2})$ attains the minimum at $l = P_L(\kappa_2\mu_2) = \kappa_2\langle\mu_1,\mu_2\rangle\mu_1$ on $\mathcal{M}_L$ by Proposition~\ref{eproj}.
    By Corollary~\ref{klcor}, the tangent vector $l - \kappa_2\mu_2$ is perpendicular to $\kappa_1\mu_1$ at $l$ with respect to $g^F$, so from the general Pythagorean theorem (Theorem~\ref{genpy}) we obtain the decomposition.
    If $\kappa_1\mu_1 = l = \kappa_2\langle\mu_1,\mu_2\rangle\mu_1$ holds, the equation follows from the formula \eqref{vmfkl} of the KL divergence.
\end{proof}

\vspace{\baselineskip}

Since the dualistic structure of the vMF manifold are invariant under orthogonal transformations of \(\mathbb{R}^n\), we may identify any \(m\)-dimensional linear subspace \(W\subset\mathbb{R}^n\) with \(\mathbb{R}^m\) via a linear isometry.
Under this identification, the dualistic structure \((g,\nabla,\nabla^*)\) induced on \(\mathcal{M}_W\) by \((g^F,\econ,\mcon)\) does not, in general, coincide with the dualistic structure on \(\mathcal{M}^m_{\mathrm{vMF}}\). 
Indeed, the potential $\Psi=\psi\circ\theta$ of $g$ with respect to $\nabla$ is given by the cumulant generating function $\psi$ restricted to $\mathbb{R}^m$, which is different from that of $\mathcal{M}^m_{\mathrm{vMF}}$.
In particular, at the origin, the respective Fisher metrics are given by
\[
g_{p_0}=\frac{1}{n}\langle\cdot,\cdot\rangle,
\qquad
g^{F,m}_{p_{0}}=\frac{1}{m}\langle\cdot,\cdot\rangle.
\]
Thus, these structures do not coincide under the natural identification when \(m<n\).
For each positive integer \(l\), let \(\mathrm{KL}^{l}\) denote the Kullback–Leibler divergence on \(\mathcal{M}^{l}_{\mathrm{vMF}}\). 
Consequently, for \(\theta_1,\theta_2\in W\simeq\mathbb{R}^m\), the divergences \(\mathrm{KL}^{n}(p_{\theta_1}^{n}\|p_{\theta_2}^{n})\) and \(\mathrm{KL}^{m}(p_{\theta_1}^{m}\|p_{\theta_2}^{m})\) need not be equal, where the superscripts distinguish the corresponding vMF families.\\

\begin{theorem}
    Let $V\subset\mathbb{R}^n$ be an $(m-1)$-dimensional linear subspace and let $\theta\in\mathbb{R}^n\setminus V$.
    Set $W=V\oplus\mathbb{R}\theta$ and identify $W$ with $\mathbb{R}^m$ via a linear isometry.
    Define
    \[
        \mathcal{M}_V^{(n)}
        =\{p_v^{n}\mid v\in V\},
        \qquad
        \mathcal{M}_V^{(m)}
        =\{p_v^{m}\mid v\in V\}.
    \]
    Then the unique minimizers of $\mathrm{KL}^n(\cdot\|p_\theta^{n})$ on $\mathcal{M}_V^{(n)}$ and $\mathrm{KL}^m(\cdot\|p_\theta^{m})$ on $\mathcal{M}_V^{(m)}$ are $p_\xi^{n}$ and $p_\xi^{m}$, respectively, where $\xi=P_V\theta$.
\end{theorem}

\begin{proof}
    By Proposition \ref{eproj}, the unique minimizer of $\mathrm{KL}^n(\cdot\|p_\theta^{n})$ on $\mathcal{M}_V^{(n)}$ is $p_\xi^{n}$, where $\xi=P_V\theta$.
    Applying Proposition \ref{eproj} to $\mathcal{M}^m_{\mathrm{vMF}}$, we obtain $p_\xi^{m}$ as the unique minimizer of $\mathrm{KL}^m(\cdot\|p_\theta^{m})$ on $\mathcal{M}_V^{(m)}$.
\end{proof}

\vspace{\baselineskip}

\section*{Acknowledgement}
The authors would like to express their sincere gratitude to Dr. Esra S. Sevim for her valuable insights, constructive feedback, and expert advice throughout the course of this work, which significantly improved the quality of this paper.

\vspace{-1pt}
\bibliography{sn-bibliography}

\end{document}